\documentclass{article}
\usepackage{amsfonts}

\newtheorem{theorem}{Theorem}

\newtheorem{claim}[theorem]{Theorem and Definition}

\newtheorem{proposition}[theorem]{Proposition}

\newenvironment{proof}[1][Proof]{\noindent\textbf{#1.} }{\ \rule{0.5em}{0.5em}}
\input{tcilatex}
\begin{document}

\title{On Stably Free Ideal Domains}
\author{Henri Bourl\`{e}s\thanks{%
Satie, ENS de Cachan/CNAM, 61 Avenue Pr\'{e}sident Wilson, F-94230 Cachan,
France (henri.bourles@satie.ens-cachan.fr) }}
\maketitle

\begin{abstract}
We define a stably free ideal domain to be a Noetherian domain whose left
and right ideals ideals are all stably free. \ Every stably free ideal
domain is a (possibly noncommutative) Dedekind domain, but the converse does
not hold. The first Weyl algebra over a field of characteristic 0 is a
typical example of stably free ideal domain. \ Some properties of these
rings are studied. \ A ring is a principal ideal domain if, and only if it
is both a stably free ideal domain and an Hermite ring.
\end{abstract}

\section{Introduction}

In a principal ideal domain (resp. a Dedekind domain), every left or right
ideal is free (resp. projective). \ An intermediate situation is the one
where every left or right ideal is stably free. \ A Noetherian domain with
this property is called a \emph{stably free ideal domain} in what follows. \
In a B\'{e}zout domain, every finitely generated (f.g.) left or right ideal
is free. \ An Ore domain in which every f.g. left or right ideal is stably
free is called a \emph{semistably free ideal domain} in what follows. \
Stably free ideal domains and semistably free ideal domains are briefly
studied in this paper.

\section{Free ideal domains and semistably free ideal domains}

\begin{claim}
\label{lemma-def-stably-free-ideal-domain}Let $\mathbf{A}$ be a ring and
consider the following conditions.\newline
(i) Every left or right ideal in $\mathbf{A}$ is stably-free.\newline
(ii) Every f.g. torsion-free $\mathbf{A}$-module is stably-free.\newline
(iii) Every f.g. left or right ideal in $\mathbf{A}$ is stably-free.\newline
(1) If $\mathbf{A}$ is a Noetherian domain, then (i)$\Leftrightarrow $(ii)$%
\Leftrightarrow $(iii). \ If these equivalent conditions hold, $\mathbf{A}$
is called a \emph{stably-free ideal domain}.\newline
(2) If $\mathbf{A}$ is an Ore domain, then (ii)$\Leftrightarrow $(iii). \ If
these equivalent conditions hold, $\mathbf{A}$ is called a \emph{%
semistably-free ideal domain}.
\end{claim}

\begin{proof}
(1) (ii)$\Rightarrow $(i): Assume that (ii) holds and let $\mathfrak{I}$ be
a left ideal in $\mathbf{A}$. \ Then $\mathfrak{I}$ is a f.g. torsion-free
module, therefore it is stably-free.

(i)$\Rightarrow $(ii): Assume that (i) holds and let $P$ be a f.g.
torsion-free $\mathbf{A}$-module. \ Since every left or right ideal is
projective, $\mathbf{A}$ is a Dedekind domain. \ Therefore, $P\cong \mathbf{A%
}^{n}\oplus \mathfrak{I}$ where $\mathfrak{I}$ is a left ideal an $n$ is an
integer (\cite{McConnell-Robson}, 5.7.8). \ Since $\mathfrak{I}$ is
stably-free, say of rank $r\geq 0$, there exists an integer $q\geq 0$ such
that $\mathfrak{I}\oplus \mathbf{A}^{q}\cong \mathbf{A}^{q+r}$. \ Therefore, 
$P\oplus \mathbf{A}^{q}\cong \mathbf{A}^{n+q+r}$ and $P$ is stably-free of
rank $n+r$. \ (i)$\Leftrightarrow $(iii) is clear.

(2) (ii)$\Rightarrow $(iii) is clear.

(iii)$\Rightarrow $(ii): If (iii) holds, $\mathbf{A}$ is semihereditary. \
Let $P$ be a torsion-free left $\mathbf{A}$-module. \ Since $\mathbf{A}$ is
an Ore domain, there exists an integer $n>0$ and an embedding $%
P\hookrightarrow \mathbf{A}^{n}$ \cite{Gentile}. \ Therefore, there exists a
finite sequence of f.g. left ideals $\left( \mathfrak{I}_{i}\right) _{1\leq
i\leq k}$ such that $P\cong \tbigoplus\nolimits_{i=1}^{k}\mathfrak{I}_{i}$ (%
\cite{Lam-Lectures}, Thm. (2.29)). \ For every index $i\in \left\{
1,...,k\right\} $, $\mathfrak{I}_{i}$ is stably-free, therefore there exist
non-negative integers $q_{i}$ and $r_{i}$ such that $\mathfrak{I}_{i}\oplus 
\mathbf{A}^{q_{i}}\cong \mathbf{A}^{q_{i}+r_{i}}$. \ As a consequence,%
\[
P\oplus \mathbf{A}^{q}\cong \mathbf{A}^{q+r} 
\]%
where $q=\tsum\nolimits_{1\leq i\leq k}q_{i}$ and $r=\tsum\nolimits_{1\leq
i\leq k}r_{i}$, and $P$ is stably-free.
\end{proof}

\section{Examples of stably free ideal domains}

The examples below involve skew polynomials.

\begin{proposition}
\label{prop-stable-free-ideal-outer-derivation}Let $\mathbf{R}$ be a
commutative stably free ideal domain.\newline
(1) Assume that $\mathbf{R}$ is a $%
\mathbb{Q}
$-algebra and let $\mathbf{A}=\mathbf{R}\left[ X;\delta \right] $ where $%
\delta $ is an outer derivation of $\mathbf{R}$ and $\mathbf{R}$ has no
proper nonzero $\delta $-stable (left or right) ideals. \ Then $\mathbf{A}$
is a stably free ideal domain.\newline
(2) Let $\mathbf{A}=\mathbf{R}\left[ X,X^{-1};\sigma \right] $ where $\sigma 
$ is an automorphism of $\mathbf{R}$ such that $\mathbf{R}$ has no proper
nonzero $\sigma $-stable (left or right) ideals and no power of $\sigma $ is
an inner automorphism of $\mathbf{R}$. \ Then $\mathbf{A}$ is a stably free
ideal domain.
\end{proposition}

\begin{proof}
The ring $\mathbf{A}$ is simple (\cite{McConnell-Robson}, 1.8.4/5),
therefore it is a noncommutative Dedekind domain (\cite{McConnell-Robson},
7.11.2), thus every left or right ideal of $\mathbf{A}$ is projective, and,
moreover, stably free (\cite{McConnell-Robson}, 12.3.3).
\end{proof}

Thus we have the following examples:

\begin{enumerate}
\item Let $k$ be a field of characteristic $0.$ \ The first Weyl algebra $%
A_{1}\left( k\right) $ and the ring $A_{1}^{\prime }\left( k\right) =k\left[
x,x^{-1}\right] \left[ X;\frac{d}{dx}\right] \cong k\left[ X\right] \left[
x,x^{-1};\sigma \right] $ with $\sigma \left( X\right) =X+1$\ (\cite%
{McConnell-Robson}, 1.8.7)\ are both stably free ideal domains.

\item Likewise, let $k=%
\mathbb{R}
$ or $%
\mathbb{C}
,$ let $k\left\{ x\right\} $ be the ring of convergent power series with
coefficients in $k,$ and let $A_{1c}\left( k\right) =k\left\{ x\right\} %
\left[ X;\frac{d}{dx}\right] .$ \ This ring is a stably free ideal domain.

\item Let $\Omega $ be a nonempty open interval of the real line and let $%
\mathcal{R}\left( \Omega \right) $ be the largest ring of rational functions
analytic in $\Omega ,$ i.e. $\mathcal{R}\left( \Omega \right) =%
\mathbb{C}
\left( x\right) \cap \mathcal{O}\left( \Omega \right) $ where $\mathcal{O}%
\left( \Omega \right) $ is the ring of all $%
\mathbb{C}
$-valued\ analytic functions in $\Omega $. \ The ring $A\left( \Omega
\right) =\mathcal{R}\left( \Omega \right) \left[ X;\frac{d}{dx}\right] $ is
a simple Dedekind domain \cite{Frohler-Oberst} and, since $\mathcal{R}\left(
\Omega \right) $ is a principal ideal domain, $A\left( \Omega \right) $ is a
stably free ideal domain.
\end{enumerate}

Note that a commutative Dedekind domain which is not a principal ideal
domain is not a stably free ideal domain (\cite{McConnell-Robson}, 11.1.5).

\section{Connection with principal ideal domains, B\'{e}zout domains, and
Hermite rings}

\begin{proposition}
(i) A ring is a principal ideal domain if, and only if it is both a stably
free ideal domain and an Hermite ring.\newline
(ii) A ring is a B\'{e}zout domain if, and only if it is both a semistably
free ideal domain and an Hermite ring.
\end{proposition}

\begin{proof}
(i): The necessary condition is clear. \ Let us prove the sufficient
condition. \ Let $\mathbf{A}$ be a stably free ideal domain and let $%
\mathfrak{a}$ be a left ideal of $\mathbf{A}$. \ This ideal is stably free.
\ If $\mathbf{A}$ is Hermite, $\mathfrak{a}$ is free, and since $\mathbf{A}$
is left Noetherian, it is a principal left ideal domain (\cite{Coh-FR},
Chap. 1, Prop. 2.2).

The proof of (ii) is similar, using (\cite{Coh-FR}, Chap. 1, Prop. 1.7).
\end{proof}

\section{Localization}

\begin{proposition}
Let $\mathbf{A}$ be a stably free ideal domain (resp. a semistably free
ideal domain) and let $S$ be a two-sided denominator set (\cite%
{McConnell-Robson}, \S 2.1). \ Then $S^{-1}\mathbf{A}$ is a stably free
ideal domain (resp. a semistably free ideal domain).
\end{proposition}

\begin{proof}
(1) Let us consider the case of stably free ideal domains. \ Let $\mathbf{A}$
be a stably free ideal domain. \ For any left ideal $\mathfrak{a}$ of $S^{-1}%
\mathbf{A}$ there exists a left ideal $\mathfrak{I}$ of $\mathbf{A}$ such
that $\mathfrak{a}=S^{-1}\mathfrak{I}.$ \ Since $\mathfrak{I}$ is stably
free, there exist integers $q$ and $r$ such that $\mathfrak{I}\oplus \mathbf{%
A}^{q}=\mathbf{A}^{r},$ therefore $S^{-1}\mathfrak{I}\oplus S^{-1}\mathbf{A}%
^{q}=S^{-1}\mathbf{A}^{r},$ and $\mathfrak{a}$ is stably free. \ The same
rationale holds for right ideals, and this proves that $S^{-1}\mathbf{A}$ is
a stably free ideal domain.

(2) The case of semistably free ideal domains is similar, considering f.g.
ideals.
\end{proof}

\end{document}